\documentclass[a4paper,12pt, british, reqno]{amsart}

\usepackage{babel}
\usepackage{mathabx}
\usepackage{amssymb}
\usepackage{amsfonts}
\usepackage{enumitem}
\usepackage{hyperref}
\usepackage[utf8]{inputenc}
\usepackage{newunicodechar}
\usepackage{mathtools}
\usepackage{varioref}
\usepackage[centering]{geometry}
\usepackage[arrow,curve,matrix]{xy}
\usepackage{csquotes}
\usepackage{dynkin-diagrams}
\usepackage{extarrows}

\usepackage{colortbl}
\usepackage{graphicx}
\usepackage{tikz}
\usepackage{tikz-cd}

\usepackage{mathrsfs}

\usepackage{multirow}
\usepackage{multicol}
\usepackage{longtable} 
\usepackage{caption}
\usepackage{hhline}
\newcolumntype{M}[1]{>{\centering\arraybackslash}m{#1}} 

\definecolor{linkred}{rgb}{0.7,0.2,0.2}
\definecolor{linkblue}{rgb}{0,0.2,0.6}

\numberwithin{figure}{section}

\setdescription{labelindent=\parindent, leftmargin=2\parindent}
\setitemize[1]{labelindent=\parindent, leftmargin=2\parindent}
\setenumerate[1]{labelindent=0cm, leftmargin=*, widest=iiii}

\DeclareFontFamily{OMS}{rsfs}{\skewchar\font'60}
\DeclareFontShape{OMS}{rsfs}{m}{n}{<-5>rsfs5 <5-7>rsfs7 <7->rsfs10 }{}
\DeclareSymbolFont{rsfs}{OMS}{rsfs}{m}{n}
\DeclareSymbolFontAlphabet{\scr}{rsfs}
\DeclareSymbolFontAlphabet{\scr}{rsfs}

\DeclareMathOperator{\Ann}{Ann}

\DeclareMathOperator{\codim}{codim}

\DeclareMathOperator{\ffi}{fi}

\DeclareMathOperator{\Pic}{Pic}
\DeclareMathOperator{\rank}{rank}

\DeclareMathOperator{\red}{red}
\DeclareMathOperator{\reg}{reg}

\DeclareMathOperator{\sing}{sing}
\DeclareMathOperator{\Spec}{Spec}

\DeclareMathOperator{\pr}{pr}

\DeclareMathOperator{\gr}{gr}

\DeclareMathOperator{\Cl}{Cl}

\newcommand{\cD}{\mathcal D}

\newcommand{\cO}{\mathcal O}

\newcommand{\bC}{\mathbb{C}}

\newcommand{\bP}{\mathbb{P}}
\newcommand{\bQ}{\mathbb{Q}}

\newcommand{\bZ}{\mathbb{Z}}

\newcommand{\bfZ}{\mathbf{Z}}
\newcommand{\bfz}{\mathbf{z}}

\newcommand{\fg}{\mathfrak{g}}
\newcommand{\fh}{\mathfrak{h}}

\newcommand{\fl}{\mathfrak{l}}
\newcommand{\fm}{\mathfrak{m}}

\newcommand{\fX}{\mathfrak{X}}

\newcommand{\aS}{{\sf S}}

\theoremstyle{plain}
\newtheorem{thm}{Theorem}[section]

\newtheorem{conjecture}[thm]{Conjecture}
\newtheorem{cor}[thm]{Corollary}
\newtheorem{defn}[thm]{Definition}

\newtheorem{lem}[thm]{Lemma}

\newtheorem{prop}[thm]{Proposition}

\theoremstyle{remark}

\newtheorem{c-n-d}[thm]{Claim and Definition}

\newtheorem{example}[thm]{Example}

\newtheorem{rem}[thm]{Remark}

\newtheorem*{rem-nonumber}{Remark}

\numberwithin{equation}{thm}

\setlist[enumerate]{label=(\thethm.\arabic*),
	before={\setcounter{enumi}{\value{equation}}},
	after={\setcounter{equation}{\value{enumi}}}}

\newcommand{\factor}[2]{\left. \raise 2pt\hbox{$#1$} \right/\hskip -2pt\raise
	-2pt\hbox{$#2$}}

\author{Baohua Fu}
\address{Baohua Fu, State Key Laboratory of Mathematical Sciences, Morningside
	Center of Mathematics, Academy of Mathematics and Systems Science, Chinese
	Academy of Sciences, Beijing 100190, China;   and School of Mathematical
	Sciences, University of Chinese Academy of Sciences, Beijing, China}
\email{\href{bhfu@math.ac.cn}{bhfu@math.ac.cn}}
\urladdr{\href{http://www.math.ac.cn/people/fbh/}{http://www.math.ac.cn/people/fbh/}}

\author{Jie Liu} %
\address{Jie Liu, Institute of Mathematics, Academy of Mathematics and Systems
	Science, Chinese Academy of Sciences, Beijing, 100190, China}
\email{\href{jliu@amss.ac.cn}{jliu@amss.ac.cn}}
\urladdr{\href{http://www.jliumath.com}{http://www.jliumath.com}}


\makeatletter
\@namedef{subjclassname@2020}{2020 Mathematics Subject Classification}
\makeatother

\title[]{Hamiltonian reductions as affine closures of cotangent bundles}
\date{\today}

\makeatletter

\hypersetup{
	pdfauthor={\authors},
	pdftitle={\@title},
	pdfsubject={\@subjclass},
	pdfkeywords={\@keywords},
	pdfstartview={Fit},
	pdfpagemode={UseOutlines},
	pdfpagelayout={OneColumn},
	colorlinks,
	linkcolor=linkblue,
	citecolor=linkred,
	urlcolor=linkred
}
\makeatother

\begin{document}
	
	\begin{abstract}
         Let $Y$ be an irreducible non-singular affine $G$-variety with a $2$-large action. We show that the Hamiltonian reduction $T^*Y/\!\!/\!\!/G$ is a symplectic variety with terminal singularities, isomorphic to the affine closure of $T^*Z_{\reg}$ where $Z \coloneqq Y/\!/G$. Furthermore, we provide sufficient conditions for the non-existence of a symplectic resolution for such varieties. These results yield three main applications: (i) providing a short proof of G.~Schwarz’s theorem on the graded surjectivity of the push-forward map $\mathcal{D}(Y)^G \to \mathcal{D}(Z)$; (ii) establishing the surjectivity of the symbol map on $Z$; and (iii) confirming the non-linear analog of a conjecture of Kaledin--Lehn--Sorger for $2$-large actions.
	\end{abstract}

	\maketitle
	\tableofcontents
	
	\section{Introduction}

    Symplectic varieties, as introduced by Beauville \cite{Beauville2000a}, are singular counterparts of hyperk\"ahler manifolds and they are playing important roles not only in algebraic symplectic geometry, but also in  geometric representation theory, 3d $\mathcal{N}=4$ supersymmetric gauge theory and so on. 
    
    Recently, we started a systematic investigation of a geometric construction of symplectic varieties through cotangent bundles. Namely for an irreducible non-singular variety $Y$ with $T^*Y$ its cotangent bundle, denote by $\cO(T^*Y)$ the ring of regular functions on $T^*Y$. Then the \emph{affine closure} $\mathcal{Z}_Y$ of $T^*Y$ is defined as the normal affine scheme ${\rm Spec}(\mathcal{O}(T^*Y))$ so that we have the natural affinization morphism 
    \[
    \psi_Y\colon T^*Y \longrightarrow \mathcal{Z}_Y\coloneqq {\rm Spec}(\mathcal{O}(T^*Y)).
    \]
    
    When $Y$ is a non-singular projective variety, it was shown in \cite{FuLiu2025a} that $\mathcal{Z}_Y$ is a symplectic variety if  the tangent bundle $T_Y$ is big and if the projective bundle $\bP T_Y$ is of Fano type.  The most interesting case is where $Y$ is quasi-affine but not affine. In this case, $\psi_Y$ is just an open embedding and, therefore,  $\mathcal{Z}_Y$ is the affine closure of $T^*Y$. It remains a challenging problem to determine when  $\mathcal{Z}_Y$ is a symplectic variety.

    Another geometric way to construct symplectic varieties is through Hamiltonian reductions, which is a powerful tool in both physics and geometry to produce new symplectic manifolds from a Hamiltonian action on a symplectic manifold.  To be more precise, let $(X,\omega)$ be a non-singular  affine symplectic variety on which there exists a Hamiltonian action by a  reductive complex linear algebraic group $G$ so that the associated moment map $\mu: X \to \fg^*$ is $G$-equivariant (see \S\,\ref{s.QHvar}). We denote by $N_X$ the scheme-theoretic fiber $\mu^{-1}(0)$, called the \emph{shell}. The \emph{Hamiltonian reduction} $X/\!\!/\!\!/G$ is by definition the affine GIT quotient:
    \[
    N_X/\!/G = \mu^{-1}(0)/\!/G\coloneqq {\rm Spec}(\mathcal{O}({\mu^{-1}(0)})^G).
    \]
    Even for the case where $X$ is a symplectic linear representation of $G$, the Hamiltonian reduction  $X/\!\!/\!\!/G$ can be non-reduced or reducible or non-normal.

    If $Y$ is an irreducible non-singular affine $G$-variety, its cotangent bundle $T^*Y$ admits a natural Hamiltonian $G$-action. Let $Y/\!/G$ be the affine GIT quotient.  Our first main result in this paper is to bridge the two geometric constructions as follows (see Definition \ref{defn.Large-action} for $2$-large actions):
	
		\begin{thm} \label{thm.main}
		Let $G$ be a reductive group and $Y$ an irreducible non-singular affine $G$-variety that is $2$-large.
		Then the following holds:
		\begin{enumerate}
			\item \label{i1.terminal} The  Hamiltonian reduction $T^*Y/\!\!/\!\!/G$ is a symplectic variety with terminal singularities.
			
			\item \label{i2.affineclosure} The  Hamiltonian reduction $T^*Y/\!\!/\!\!/G$ is  the affine closure of $T^*(Y/\!/G)_{\rm reg}$. In particular, if $Y/\!/G$ is non-singular, then $T^*Y/\!\!/\!\!/G \cong T^*(Y/\!/G)$.
			
			\item \label{i3.singularlocus} The  Hamiltonian reduction $T^*Y/\!\!/\!\!/G$ is non-singular if and only if $Y/\!/G$ is non-singular.
		\end{enumerate}
	\end{thm}

    The $2$-largeness condition in Theorem \ref{thm.main} is not overly restrictive: as shown in \cite[Corollary 11.6 (2)]{Schwarz1995} (see also \cite[Theorem 3.6]{HerbigSchwarzSeaton2020}), most $G$-modules are in fact $2$-large, with only finitely many exceptions under mild assumptions. On the other hand, for the special case in which $Y$ is a linear representation of $G$, part \ref{i1.terminal} of Theorem \ref{thm.main} was proved in \cite{HerbigSchwarzSeaton2020} (see also \cite{HerbigSchwarzSeaton2024a,HerbigSchwarzSeaton2024} for other related results).

    As an application of Theorem~\ref{thm.main}, we investigate the ring of differential operators on the quotient $Y/\!/G$. To state our results precisely, we recall the following notation (see also \S\,\ref{s.DO} for further background). For an irreducible normal variety $W$, let $\cD(W)$ denote the ring of global differential operators on $W$, and define the graded $\bC$-algebra
\[
S(W)\coloneq \bigoplus_{p\geq 0} H^0(W_{\reg},\aS^p T_{W_{\reg}}) = \cO(T^*W_{\reg}).
\]

	\begin{thm}
		\label{thm.Diff-Quot}
		Let $G$ be a reductive group and $Y$ an irreducible non-singular affine $G$-variety that is $2$-large. Set $Z\coloneqq Y/\!/G$ with the quotient map $\pi_Y\colon Y\rightarrow Z$.
		\begin{enumerate}
			\item \cite[Theorem 0.4]{Schwarz1995} The canonical morphism $(\pi_Y)_*\colon \cD(Y)^G \rightarrow \cD(Z)$  is graded surjective.
			
			\item The symbol map $\sigma_Z \colon \cD(Z)\to S(Z)$ is surjective.
		\end{enumerate}
	\end{thm}
    
    We next investigate the ($\bQ$-)factoriality of the Hamiltonian reduction $T^*Y/\!\!/\!\!/G$ and use this to deduce the non‑existence of symplectic resolutions under certain conditions.  Recall that the \emph{abelianization} of $G$ is $G_{\operatorname{ab}}\coloneqq G/[G,G]$. A group $G$ is called \emph{perfect} if $G_{\operatorname{ab}}=\{e\}$, i.e., $G=[G,G]$. 
	
	\begin{thm}
		\label{thm.factoriality}
		Let $G$ be a reductive group and $Y$ an irreducible non-singular affine $G$-variety that is $2$-large. Assume that $\cO(Y)^{\times}=\bC^{\times}$ or $G$ is connected.
		\begin{enumerate}
			\item If $G_{\operatorname{ab}}$ is finite and $\Pic(Y)$ is finite, then $\Cl(T^*Y/\!\!/\!\!/G)$ is finite. In particular, $T^*Y/\!\!/\!\!/G$ is $\bQ$-factorial.
			
			\item If $G$ is perfect and $\Pic(Y)$ is trivial, then $\Cl(T^*Y/\!\!/\!\!/G)$ is trivial; that is, $T^*Y/\!\!/\!\!/G$ is factorial.
		\end{enumerate}
	\end{thm}
	
	We present a partial converse of Theorem \ref{thm.factoriality} in Corollary \ref{cor.Conv-Thm-factoriality}. For $Y$ a $3$-large linear representation of $G$, the local factoriality and $\bQ$-factoriality of $T^*Y/\!\!/\!\!/G$ were already established in \cite[Theorem 1.2]{BellamySchedler2019}. By combining Theorem \ref{thm.main} and Theorem \ref{thm.factoriality}, we obtain the following result.
    
    \begin{thm}
    \label{thm.KLSconj-2large}
        Let $G$ be a semisimple group, and let $Y$ be an irreducible non-singular affine $G$-variety which is $2$-large. Assume that $\Pic(Y)$ is finite. Then $T^*Y/\!\!/\!\!/G$ admits a symplectic resolution if and only if $Y/\!/G$ is non-singular.
    \end{thm}

     Theorem \ref{thm.KLSconj-2large} yields a positive answer  to a non-linear analog for all $2$-large actions of the following conjecture of D.~Kaledin, M.~Lehn and C.~Sorger, which in the linear case follows from \cite[Corollary 1.3]{BellamySchedler2019}.
    
    \begin{conjecture}\label{conj.KLS}
        Let $G$ be a semisimple group and let $V$ be a finite dimensional representation of $G$. If there exists a symplectic resolution of $T^*V/\!\!/\!\!/G$, then $V/\!/G$ is non-singular.
    \end{conjecture}

    \subsection*{Acknowledgments}
	   Both authors are supported by the National Key Research and Development Program of China (No. 2025YFA1017302),  the CAS Project for Young Scientists in Basic Research (No. YSBR-033) and the NSFC grant (No. 12288201). J.~Liu is supported  by the Youth Innovation Promotion Association CAS. We thank G.~Bellamy and G.~Schwarz for comments on a first version of this text.
 
	\section{Quasi-Hamiltonian varieties}
    \label{s.QHvar}
	
	Throughout this paper, we work over complex numbers, and $G$ always denotes a reductive complex linear algebraic group. By varieties we mean separated reduced schemes of finite type over $\bC$, which are not necessarily irreducible.

	\subsection{Hamiltonian actions}

	\begin{defn}
		Let $X$ be an irreducible normal variety.
		\begin{enumerate}
			\item We call $X$ quasi-symplectic if there exists a regular symplectic
			form $\omega$ on its non-singular locus $X_{\reg}$.
			
			\item We call $X$ symplectic if it is quasi-symplectic and for any resolution
			$f\colon Z\rightarrow X$, the form $f^*\omega$ extends to a regular $2$-form
			on $Z$.
		\end{enumerate}
	\end{defn}
	
	Let $X$ be a non-singular symplectic variety equipped with an action of $G$ by symplectomorphisms. For a regular function $H\in \cO(X)$, the \emph{skew-gradient} $V_H$ of $H$ is, by definition, the regular vector field on $X$ given by the following equality:
	\[
	\omega(V_H,\cdot) = dH.
	\]
	To any element $\xi\in \fg$ one associates the velocity vector field $V_{\xi}$.
	
	\begin{defn}
		The action $G$ on $X$ equipped with a linear map $\fg\rightarrow \cO(X)$, $\xi \rightarrow H_{\xi}$, is said to be Hamiltonian if it satisfies the following two conditions:
		\begin{enumerate}
			\item[(H1)] The linear map $\fg\rightarrow \cO(X)$ is $G$-equivariant.
			
			\item[(H2)] $V_{H_{\xi}}=V_{\xi}$.
		\end{enumerate}
		Then $X$ is called a Hamiltonian $G$-variety so that $V_{\xi}$ and $H_{\xi}$ are called the Hamiltonian vector field and the Hamiltonian function associated with $\xi$, respectively.
	\end{defn}
	
	For a Hamiltonian action $G$ on $X$, we can define a morphism $\mu_{G,X}\colon X\rightarrow \fg^*$ using the following formula:
	\[
	\langle \mu_{G,X}(x),\xi\rangle \coloneqq H_{\xi}(x), \quad \xi\in \fg, \quad x\in X.
	\]
	This morphism is called the \emph{moment map} of the Hamiltonian $G$-variety $X$ so that $\mu_{G,X}$ is $G$-equivariant with respect to the coadjoint action of $G$ on $\fg^*$ such that $H_{\xi}\in \cO(X)$ is nothing but the composition
	\[
	X \xlongrightarrow{\mu_{G,X}} \fg^* \xlongrightarrow{\xi} \bC.
	\]
	Moreover, if $G$ and $X$ are clear from the context, we will usually abbreviate $\mu_{G,X}$ to $\mu$. 
	
    \begin{defn}
    	A (quasi-)Hamiltonian $G$-variety is a (quasi-)symplectic variety $X$ equipped with a symplectic $G$-action such that the restricted $G$-action to $X_{\reg}$ is Hamiltonian.
    \end{defn}

    \begin{rem}
        The term ``quasi-Hamiltonian $G$-spaces'' already appears in the literature \cite{AlekseevMalkinMeinrenken1998} (in a different context); we emphasize that the notion introduced here is very different from that earlier work.
    \end{rem}
    
    Given a quasi-Hamiltonian $G$-variety, since $X$ is normal and $\codim(X\setminus X_{\reg})\geq 2$, the moment map $\mu_{G,X_{\reg}}\colon X_{\reg}\rightarrow \fg^*$ extends uniquely to a $G$-equivariant morphism $X\rightarrow \fg^*$, again denoted by $\mu_{G,X}$. We shall call $\mu_{G,X}$ also the \emph{moment map} of the quasi-Hamiltonian $G$-variety $X$.
    
    \begin{example}
    	Let $Y$ be an irreducible non-singular $G$-variety, $X\coloneqq T^*Y$ the cotangent bundle of $Y$. Then $X$ is a non-singular symplectic variety with the canonical symplectic form $d\eta$, where $\eta$ is the Liouville $1$-form on $T^*Y$. Moreover, the lifted action of $G$ on $X$ is Hamiltonian, and the moment map is given by
    	\[
    	\mu_{G,X}((y,\alpha),\xi) \coloneqq \langle \alpha,V_{\xi}(y)\rangle,
    	\]
    	where $y\in Y$, $\alpha\in T_y^* Y$ and $\xi\in \fg$. More generally, singular affine Hamiltonian $G$-varieties can also be constructed by considering the affine closure of cotangent bundles. We refer the reader to \cite{FuLiu2025,FuLiu2025a} for examples.
    \end{example}
	\subsection{The shell}
	
	Let $X$ be a quasi-Hamiltonian $G$-variety. The \emph{shell} $N_{X}$ is defined to be the scheme-theoretic fiber of the moment map $\mu\colon X\rightarrow \fg^*$ over $0\in \fg^*$. We denote by $N_{X,\ffi}$ the subset of $N_X$
	consisting of points $x\in N_X\cap X_{\reg}$ such that the isotropy subgroup $G_x$ is finite. If $X$
	is clear from the context, we will simply denote $N_X$ by
	$N$ and $N_{X,\ffi}$ by $N_{\ffi}$. The following result is well known for experts, we include a complete proof for the reader's convenience. See also \cite[Lemmas 3.2 and 3.3]{KaledinLehnSorger2006} and \cite[\S\,2.3]{HerbigSchwarzSeaton2024}.
	
	\begin{lem}
		\label{lem.Prop-of-Shell}
		Let $X$ be a Cohen-Macaulay quasi-Hamiltonian $G$-variety such that $N_{\ffi}$ is non-empty. Let $n=\dim(X)$ and $d=\dim \fg$.
		\begin{enumerate}			
			\item\label{i1.Prop-of-Shell} If $\dim(N\setminus
			N_{\ffi})\leq n-d$, then $N$ is a complete intersection, and hence Cohen-Macaulay. In particular $N$ is non-singular at $x \in X$ if and only if $x\in N_{\ffi}$.
			
			\item\label{i2.Prop-of-Shell} If $\dim(N\setminus
			N_{\ffi})\leq n-d-1$, then $N$ is reduced.
			
			\item\label{i3.Prop-of-Shell} If $\dim(N\setminus
			N_{\ffi})\leq n-d-2$, then $N$ is normal.
		\end{enumerate}
	\end{lem}
	
	\begin{proof}
		For any point $x\in X$, Condition (H2) implies that for any $v\in T_x X$ and any $\xi \in \fg$, we have
		\[
		\langle d_x\mu(v),\xi\rangle = \omega_x(V_{\xi}(x),v).
		\]
		Note that $\xi \in \fg_x$ if and only if $V_\xi(x)=0$, and the latter is equivalent to $\langle d_x\mu(v),\xi\rangle=0$ for all $v \in T_x X$. It follows that the image of 
        \[
        d_x\mu: T_x X \longrightarrow T_{\mu(x)}\fg^* \cong \fg^*
        \]
        is  $(\fg/\fg_x)^* \subset \fg^*$. Thus,  
        the tangent map $d_x\mu$ is surjective if and only if $G_x$ is finite. In particular, since $N_{\ffi}$ is non-empty, the map $\mu$ is dominant.
		
		For \ref{i1.Prop-of-Shell}, since $\mu$ is dominant, the shell $N$ is a complete intersection if and only if $N$ has pure dimension $n-d$ by \cite[Theorem 8.21A(c)]{Hartshorne1977}. By our assumption and Krull's Hauptidealsatz,	it remains to prove that $\dim(N_{\ffi})=n-d$. Note that $X$ is non-singular along $N_{\ffi}$ and $\mu$ is smooth along $N_{\ffi}$, it follows from \cite[Proposition 10.4]{Hartshorne1977} that $N_{\ffi}$ is non-singular of pure dimension $n-d$. So $N$ is Cohen-Macaulay by \cite[Theorem 8.21A(d)]{Hartshorne1977}. Now, it follows from the Jacobian criterion that $N$ is non-singular at $x$ only if $X$ is non-singular at $x$ and hence if and only if $x\in N\cap X_{\reg}$ and $\mu$ is smooth at $x$, which is equivalent to $x\in N_{\ffi}$.
		
		For \ref{i2.Prop-of-Shell}, since the open subset $N_{\ffi}$ is dense in $N$ by assumption, the shell $N$ is generically reduced and therefore reduced as it is Cohen-Macaulay.
		
		For \ref{i3.Prop-of-Shell}, it follows from the assumption that $N$ is Cohen-Macaulay with $\codim(N_{\sing})\geq 2$, so $N$ is normal by Serre's criterion \cite[Theorem 8.22A]{Hartshorne1977}.
	\end{proof}

	\section{Hamiltonian reduction}
    \label{s.Hamilred}

    \subsection{Hamiltonian slice theorem}

     For a scheme $Z$, we denote by $Z^{\red}$ its associated reduced scheme. Let $Y$ be an affine scheme over $\Spec\bC$ equipped with a $G$-action. We denote by $Y/\!/G\coloneqq \Spec(\cO(Y)^G)$ the invariant-theoretical quotient with the natural projection $\pi_Y\colon Y\rightarrow Y/\!/G$. A subset of $Y$ is \emph{$G$-saturated} if it is a union of fibers of $\pi_Y$. We recall the following basic definition:
	 
	 \begin{defn}
	 	\label{defn.Exc-Mor}
	 	Let $X$, $Y$ be affine $G$-varieties. An excellent morphism is a  $G$-equivariant morphism $\varphi\colon
	 	Y\rightarrow X$ such that the following holds:
	 	\begin{enumerate}
	 		\item The induced morphism $\varphi/\!/G\colon Y/\!/G \rightarrow X/\!/G$ is
	 		\'etale;
	 		
	 		\item The morphism $(\pi_Y,\varphi)\colon Y\rightarrow Y/\!/G\times_{X/\!/G}
	 		X$ is an isomorphism.
	 	\end{enumerate}
	 	In particular, $\varphi\colon Y\rightarrow X$ is also \'etale, because a
	 	base change of an \'etale map is \'etale.
	 	\[
	 	\begin{tikzcd}[row sep=large, column sep=large]
	 		Y \arrow[r,"{(\pi_Y,\varphi)}","\cong"'] \arrow[d,"\pi_Y"]
	 		& Y/\!/G\times_{X/\!/G} X \arrow[r] \arrow[d,"\pi"]
	 		& X \arrow[d,"\pi_X"] \\
	 		Y/\!/G \arrow[r,equal]
	 		& Y/\!/G \arrow[r,"{\varphi/\!/G}"]
	 		& X/\!/G
	 	\end{tikzcd}
	 	\]
	 \end{defn}
	 
   We require the following Hamiltonian slice theorem of Jung \cite{Jung2009}, which is the symplectic analog of Luna's slice theorem. A simplified proof of this theorem can be found in \cite[\S\,6]{BuloisLehnLehnTerpereau2017} (see also \cite[Theorem 3.15]{HerbigSchwarzSeaton2020} for the linear case). Although the statement in \cite[Theorem 6.1]{BuloisLehnLehnTerpereau2017} differs slightly from the one below, the same argument yields our version. For completeness, we include a detailed proof in Appendix \S\,\ref{s.HSThm} with some minor adaptations to the details.
	
	\begin{thm}
		\label{thm.Symplectic-Slice}
		Let $X$ be an affine quasi-Hamiltonian $G$-variety  equipped with a symplectic form $\omega$ on $X_{\reg}$. Let $O_x=Gx$ be a closed orbit with $x\in N_X\cap X_{\reg}$. Then there exists a non-singular locally closed affine $G_x$-stable subset $U\subset X_{\reg}$ containing $x$ such that $\omega|_U$ is symplectic and the morphism 
		\[
		\psi\colon G\times_{G_x} N^{\red}_U \longrightarrow N^{\red}_X, \quad (g,u)\longmapsto gu
		\] 
		is excellent, where $N_U$ is the shell of the Hamiltonian $G_x$-variety $U$.
	\end{thm}
	
	The locally closed subset $U\subset X_{\reg}$ with the properties of the theorem will be called a \emph{Hamiltonian slice at $x$}.
	
	\subsection{Isotropy-type stratification}
	
	Let $Y$ be an affine $G$-variety (not necessarily irreducible). For $r\in \bZ_{\geq 0}$, we denote by $Y_{(r)}$ the subset of $Y$ consisting of points with isotropy groups of dimension $r$, which is a locally closed subset by the upper semi-continuity of $\dim(G_y)$, $y\in Y$. 
	
	We denote by $(G_y)$ the \emph{conjugacy class} of $G_y$, $y\in Y$, which we call an \emph{isotropy class} of $Y$, and an isotropy class is called \emph{closed} if the orbit $O_y=Gy$ is closed in $Y$. Then a closed isotropy class is reductive by Matsushima's theorem, and there are only finitely many closed isotropy classes of $Y$ by Luna's slice theorem \cite{Luna1973}. For a closed isotropy class $(H)$ of $G$, we denote by $(Y/\!/G)_{(H)}$ the points in $Y/\!/G$ corresponding to closed orbits with isotropy class in $(H)$. This is a constructible subset of $Y/\!/G$ by Luna's slice theorem. Let $Y^{(H)}$ be the inverse image of $(Y/\!/G)_{(H)}$ in $Y$.
	
	Assume in addition that $Y/\!/G$ is irreducible, then there is a unique closed isotropy class $(H)$, called the \emph{principal isotropy class}, such that $(Y/\!/G)_{\pr}\coloneqq (Y/\!/G)_{(H)}$ is dense in $Y/\!/G$ and the closed orbits $O_y$ with $G_y\in (H)$ are called \emph{principal orbits}. Then $(Y/\!/G)_{\pr}\subset Y/\!/G$ is called the \emph{principal stratum} and is open in $Y/\!/G$ by Luna's slice theorem. Let $Y_{\pr}\coloneqq \pi_Y^{-1}((Y/\!/G)_{\pr})$. We say that $Y$ has \emph{finite principal isotropy groups (FPIG)} if its principal isotropy groups are finite.

	\subsection{Hamiltonian reduction}
	

	\begin{defn}
		The Hamiltonian reduction of an affine Hamiltonian $G$-variety $X$ is defined as
		\[
		X/\!\!/\!\!/G \coloneqq N_X/\!/G.
		\]
	\end{defn}    
	
	Now we apply Theorem \ref{thm.Symplectic-Slice} to study the isotropy-type stratifications of Hamiltonian reductions. 
	
	\begin{lem}
		\label{lem.Slice-iso-stratum}
		Let $X$ be an affine Hamiltonian $G$-variety. Let $O_x=Gx\subset N_X$ be a closed orbit with $x\in N_X\cap X_{\reg}$ . Let $x\in U$ be a Hamiltonian slice. Then there exists an open subset $V$ of $N_X^{\red}/\!/G$ containing $[O_x]$ such that
		\[
		\dim\left(V\cap (N_X^{\red}/\!/G)_{(G_x)}\right) = \dim (N_U^{\red}/\!/G_x)_{(G_x)}.
		\]
	\end{lem}
	
	\begin{proof}
		By Theorem \ref{thm.Symplectic-Slice} and \cite[Proposition 5.3(2)]{BrionKraftSchwarz2025}, there exists an \'etale morphism
		\[
		\psi/\!/G\colon N_U^{\red}/\!/G_x\cong (G\times_{G_x} N_U^{\red})/\!/G \longrightarrow N_X^{\red}/\!/G.
		\]
		Denote by $V$ the image of $\psi/\!/G$. Then $V$ is open and contains $[O_x]$. Moreover, by the second condition of excellent morphisms, we have
		\[
		(\psi/\!/G)^{-1} \left((N_X^{\red}/\!/G)_{(G_x)}\right) = \left((G\times_{G_x} N_U^{\red})/\!/G\right)_{(G_x)} = \left(N_U^{\red}/\!/G_x\right)_{(G_x)},
		\]
		from which the required equality follows.
	\end{proof}    
	
	\begin{lem}
		\label{lem.dim-iso-stratum}
		Let $X$ be an affine Hamiltonian $G$-variety. Let $O_x=Gx\subset N_X$ be a closed orbit with $x\in N_X\cap X_{\reg}$ . Let $x\in U$ be a Hamiltonian slice. Then we have
		\[
		\dim (N_U^{\red}/\!/G_x)_{(G_x)}= \dim(T_x X)^{G_x} - 2(\dim {\rm Norm}_G(G_x) - \dim G_x), 
		\] where ${\rm Norm}_G(G_x)$ is the normalizer of $G_x$ in $G$.
	\end{lem} 
	
	\begin{proof}
		Note that $U^{G_x}$ is contained in $N_U$ and the induced morphism 
		\[
		U^{G_x}\longrightarrow (N_U^{\red}/\!/G_x)_{(G_x)}
		\]
		is bijective, so one gets
		\[
		\dim (N_U^{\red}/\!/G_x)_{(G_x)} = \dim \left(N_U^{\red}\right)^{G_x} = \dim U^{G_x}.
		\]
		On the other hand, by Corollary \ref{cor.Hamiltonian-slice-linear}, we have
		\[
		\dim(U^{G_x}) = \dim(T_x U)^{G_x} = \dim(T_x X)^{G_x} - \dim(\fg/\fg_x)^{G_x} - \dim((\fg/\fg_x)^*)^{G_x}.
		\]
		In particular, since
		\[
		\dim(\fg/\fg_x)^{G_x}=\dim((\fg/\fg_x)^*)^{G_x} = \dim {\rm Norm}_G(G_x) - \dim(G_x),
		\]
		the lemma follows.
	\end{proof}
	
	\begin{lem}
		\label{lem.FPIG-pri-stratum}
		Let $X$ be a non-singular affine Hamiltonian $G$-variety, equipped with a symplectic form $\omega$. Assume that $N$ is reduced and has FPIG such that $N/\!/G$ is irreducible. Then $(N/\!/G)_{\pr} \subset (N/\!/G)_{\reg}$ and there exists a natural symplectic form on $(N/\!/G)_{\pr}$.
	\end{lem} 
	
	\begin{proof}
		Let $O_x=Gx\subset N$ be a principal orbit and let $x\in U\subset X$ be a Hamiltonian slice. Since $G_x$ is finite, we have $N_U=U$. Consider the excellent $G_x$-equivariant morphism $U\rightarrow T_x U$ (see Corollary \ref{cor.Hamiltonian-slice-linear}). Since $(G_x)$ is the principal isotropy class, the set of points in $U$ with stabilizer $G_x$ is open and dense in $U$, in particular, $U^{G_x}$ is open and dense. This implies that the morphisms appearing in the following commutative diagram are dominant:
		\[
		\begin{tikzcd}[row sep=large,column sep=large]
			U^{G_x} \arrow[r] \arrow[d]      &   (T_x U)^{G_x}  \arrow[d] \\
			U/\!/G_x \arrow[r]               &    (T_x U)/\!/G_x.
		\end{tikzcd}
		\] 
		So $\dim(T_x U)^{G_x} = \dim T_x U$ and hence $G_x$ acts trivially on $T_xU$ and $U$. It follows that $U/\!/G_x = U$ is non-singular and carries a natural symplectic form. Finally, we observe that there exists a natural symplectic linear isomorphism
		\[
		T_x U \longrightarrow (T_x O_x)^{\perp_{\omega_x}}/T_{x} O_x,
		\]
		so, the symplectic form is independent of the choice of the Hamiltonian slice $U$. 
	\end{proof}   
	
	The following criterion allows us to determine when a Hamiltonian reduction is a symplectic variety. 

    \begin{prop} \label{prop.codimsymplectic}
    	Let $X$ be a non-singular affine Hamiltonian $G$-variety. Assume that
    	\begin{enumerate}
    		\item $N_{\ffi}$ is non-empty, 
    		\item $\dim(N \setminus N_{\ffi}) \leq \dim(X) -\dim(G)-2$, 
    		\item $N^{\red}$ has FPIG, 
    		\item $N^{\red}/\!/G$ is irreducible,
    		\item\label{i5.prop-codimsymplectic} for any $x\in N^{\red}\setminus N^{\red}_{\pr}$ with closed $G$-orbit, we have
    		$$
    		\codim (T_xX)^{G_x} \geq 2 (\dim G - \dim {\rm Norm}_G(G_x) + \dim G_x)+4. 
    		$$
    	\end{enumerate}
    	Then $X/\!\!/\!\!/G$ is a symplectic variety with terminal singularities.
    \end{prop}
    
    \begin{proof}
    	By Lemma \ref{lem.Prop-of-Shell}, the shell $N$ is normal, so $N=N^{\red}$ and consequently $N/\!/G$ is also normal.  In particular, since $N$ has FPIG by assumption, the open subset $N_{\pr}$ consists precisely of the closed principal orbits by Luna's slice theorem. This implies
    	\[
    	\dim(X/\!\!/\!\!/G) = \dim(N/\!/G) = \dim N - \dim G = \dim X - 2\dim G.
    	\]
    	Now, let $(H)$ be a closed non-principal isotropy class. By Lemmas \ref{lem.Slice-iso-stratum}, \ref{lem.dim-iso-stratum} and the assumption, we have
    	\[
    	\dim(N/\!/G)_{(H)} \leq \dim X - 2\dim G -4 \leq \dim(N/\!/G) - 4.
    	\]
    	This implies that $(N/\!/G)\setminus (N/\!/G)_{\pr}$ has codimension at least four, and the canonical symplectic form on $(N/\!/G)_{\pr}$ obtained in Lemma \ref{lem.FPIG-pri-stratum} then extends to a $2$-form on $(N/\!/G)_{\reg}$, which is again symplectic because its degenerate locus has codimension one if it is non-empty. Finally, since $(N/\!/G)\setminus (N/\!/G)_{\reg}\subset (N/\!/G)\setminus (N/\!/G)_{\pr}$ has codimension at least four, it follows that $X/\!\!/\!\!/G=N/\!/G$ is a symplectic variety by Flenner's theorem \cite{Flenner1988} and with terminal singularities by \cite{Namikawa2001a}.
    \end{proof}    
    
    \begin{rem}
    A modified version of Proposition \ref{prop.codimsymplectic} also holds for Cohen–Macaulay affine quasi-Hamiltonian $G$-varieties. The main difficulty in the singular setting is that the Hamiltonian slice theorem does not hold at singular points, and we have no control on the dimension of $(N/\!/G)_{(H)}$ at closed orbits lying in the locus $N\cap X_{\sing}$. However, if the image $\pi_N(N\cap X_{\sing})$ has codimension at least four in $N/\!/G$ (for example, when $\dim X_{\sing}\leq \dim X-2\dim G-4$), then Proposition \ref{prop.codimsymplectic} remains valid when we restrict it to points $x\in (N^{\red}\cap X_{\reg})\setminus N_{\pr}^{\red}$. We leave the verification of these details to the interested reader.
    \end{rem}       

    \section{Hamiltonian reduction of cotangent bundles}
    \label{s.Hamilredcotangent}

    We now apply Proposition \ref{prop.codimsymplectic} to study the Hamiltonian reduction of the cotangent bundle of a smooth affine variety with a $2$-large $G$-action, thus proving Theorems \ref{thm.main} and \ref{thm.factoriality}.

    Throughout this section, we fix the following notation. Given an irreducible non-singular affine $G$-variety $Y$, we denote by $\mu\colon T^*Y\rightarrow \fg^*$ the canonical moment map. The shell of $\mu$ is denoted by $N$. The GIT quotient of $Y$ under $G$ is written as $Z\coloneqq Y/\!/G$, with the quotient map $\pi_Y\colon Y\rightarrow Z$. The Hamiltonian reduction of $T^*Y$ is denoted as $M\coloneqq T^*Y/\!\!/\!\!/G=N/\!/G$ with the corresponding quotient $\pi_N\colon N\rightarrow M$.
    
    \subsection{Large actions}
    
    We start with some basic definitions.
	
	\begin{defn}
    \label{defn.Large-action}
		Let $Y$ be an irreducible affine $G$-variety and $k\in \bZ_{\geq 0}$. Then $Y$ is called
		\begin{enumerate}			
			\item $k$-principal if $\codim(Y\setminus Y_{\pr})\geq k$, and
			
			\item $k$-modular if $\codim(Y_{(r)})\geq r+k$ for $1\leq r\leq \dim(G)$, and
			
			\item $k$-large if $Y$ is $k$-principal, $k$-modular and has FPIG.
		\end{enumerate}
	\end{defn}
	
	The following result is a non-linear generalization of \cite[Proposition 3.2]{HerbigSchwarzSeaton2020}.
	
	\begin{prop}
		\label{prop.Shell-k-modular}
		Let $Y$ be an irreducible non-singular affine $G$-variety. Let $\dim(Y)=n$, $\dim(G)=d$ and $k\in \bZ_{\geq 0}$.
		\begin{enumerate}
			\item If $Y$ is $k$-modular, then $N_{\ffi}$ is non-empty and $\dim(N\setminus N_{\ffi})\leq 2n-d-k$. In particular, the shell $N$ is a Cohen-Macaulay complete intersection.
			
			\item If $Y$ is $1$-modular, then $N$ is irreducible and reduced.
			
			\item If $Y$ is $2$-modular, then $N$ is irreducible and normal.
		\end{enumerate}
	\end{prop}
	
	\begin{proof}
		For points $y\in Y$ and $w\in T_y^*Y$, we have $G_w\subset G_{y}$. In particular, if $Y$ is $k$-modular, then $Y_{(0)}$ is non-empty and consequently $N_{\ffi}$ is non-empty. On the other hand, note that the restricted map 
		\[
		\mu|_{T_y^* Y}\colon T_y^* Y\longrightarrow \fg^*
		\]
		is exactly given by the cotangent map at $y$ of the orbit map $\operatorname{ev}_y\colon G\rightarrow Gy\eqqcolon O_y\subset Y$, $g\mapsto gy$; that is, the dual of the tangent map
		\[
		d_e\operatorname{ev}_y\colon T_e G=\fg\longrightarrow T_y O_y \subset T_y Y.
		\]
		In particular, the intersection $N\cap T_y^*Y$ is a linear subspace of $T_y^* Y$ with codimension $\dim O_y=d-\dim(G_y)$. Let $p\colon T^*Y\rightarrow Y$ be the natural projection. Set $N^r=p^{-1}(Y_{(r)})\cap N$. Then, for $1\leq r\leq \dim G$, we have
		\begin{equation}
			\label{eq.dim-Nr}
			\dim N^r \leq \dim Y_{(r)} + (n-d+r) \leq 2n-d-k.
		\end{equation}
		In particular, since $N_{\ffi}$ contains $N^0$, one obtains
		\[
		\dim(N\setminus N_{\ffi}) \leq \max_{1\leq r\leq d} \dim N^r \leq 2n-d-k.
		\]
		Now, by Lemma \ref{lem.Prop-of-Shell}, it remains to prove the irreducibility of $N$ for $k\geq 1$. Note that the fibers of the projection
		\[
		N^0 \longrightarrow Y_{(0)}
		\]
		is a codimension $d$ linear subspace of $T^*_y Y$. In particular, since $Y_{(0)}$ is irreducible, the variety $N^0$ is also irreducible. Now we observe
		\[
		\dim(N\setminus N^{0}) \leq \max_{1\leq r\leq d} \dim N^r \leq 2n-d-k \leq 2n-d-1<\dim N.
		\]
		Hence $N$ is irreducible as it has pure dimension $2n-d$.
	\end{proof}

	\begin{lem}[\protect{cf. \cite[\S\,4.1]{HerbigSchwarzSeaton2020}}]
		\label{lem.Ham-cotangent}
		Let $Y$ be an irreducible non-singular $2$-large affine $G$-variety. 
		\begin{enumerate}
			\item\label{i1.lem-2-large-N} Then $N$ is an irreducible normal variety such that $G_x=G_y$ for any $x\in N$ with $p(x)\in Y_{\pr}$, where $p\colon T^*Y\rightarrow Y$ is the natural projection. In particular, $N$ has FPIG and satisfies
            \[
			p^{-1}(Y_{\pr})\cap N\subset N_{\pr}.
			\] 
			
			\item\label{i2.lem-2-large-N} For any $x\in N\setminus N_{\pr}$ with closed $G$-orbit, we have
			\[
			\codim (T_x (T^*Y))^{G_x} \geq 2 (\dim G - \dim {\rm Norm}_G(G_x) + \dim G_x)+4.
			\]
		\end{enumerate}
	\end{lem}
	
	\begin{proof}
		Firstly, we note that $N$ is an irreducible normal variety by Proposition \ref{prop.Shell-k-modular}. Moreover, since $Y$ has FPIG, to prove \ref{i1.lem-2-large-N}, it suffices to show that any point $x\in p^{-1}(Y_{\pr})\cap N$ has a closed $G$-orbit such that $G_x=G_y$.
        
        Let $x\in p^{-1}(Y_{\pr})\cap N$ and $y\coloneqq p(x)\in Y_{\pr}$. Moreover, since $Y$ has FPIG, the open subset $Y_{\pr}$ consists of closed principal orbits, so $O_y=Gy$ is closed in $Y$ and $G_x\subset G_y$ is a finite subgroup. Moreover, the orbit $O_x=Gx$ is contained in the closed subset 
		\[
		p^{-1}(O_y) \cong G\times_{G_y} T_y^* Y
		\]
		of $X$ such that $O_x\cong G\times_{G_y} O'_x$,  where $G_y$ acts on $T^*_y Y$ by the isotropy action and the orbit $O'_x=G_y x\subset T_y^* Y$ is a finite set. So $O_x$ is closed in $p^{-1}(O_y)$ and is therefore closed in $T^*Y$.  
        
        Since $G_y$ is finite and $y\in Y_{\pr}$, it follows from Luna's slice theorem that there exists an \'etale slice $S\subset Y$ at $y$ such that $G_y$ acts on $S$ trivially (cf. Proof of Lemma \ref{lem.FPIG-pri-stratum} and \cite[Theorem 5.3]{BrionKraftSchwarz2025}). In particular, $G_y$ acts on $T_y S$ trivially.  In particular, as $N\cap T_y^* Y\cong T_y^*S$ as $G_y$-modules, one derives $G_x=G_y$, hence the statement \ref{i1.lem-2-large-N} follows.
		
		For \ref{i2.lem-2-large-N}, consider any $x\in N \setminus N_{{\rm pr}}$ with closed orbit so that $y=p(x)\not\in Y_{\pr}$. Following the arguments in \cite[Lemmas 4.1, 4.2 and 4.3]{HerbigSchwarzSeaton2020}, we have
		\[
		\dim Y^{G_x}\leq \dim Y-\dim G + \dim {\rm Norm}_G(G_x) - \dim G_x - 2,
		\]
		which then implies
		\begin{align*}
			\codim (T_x (T^*Y))^{G_x} &  = 2(\dim Y - \dim Y^{G_x}) \\
			                     &  \geq 2 (\dim G - \dim {\rm Norm}_G(G_x) + \dim G_x)+4.
		\end{align*}
		as desired.
	\end{proof}

	\begin{proof}[Proof of Theorem \ref{thm.main}]
		
		The statement \ref{i1.terminal} follows from Proposition \ref{prop.codimsymplectic}, Proposition \ref{prop.Shell-k-modular} and Lemma \ref{lem.Ham-cotangent}. For simplicity, we set $N^{\diamond}=p^{-1}(Y_{\pr}) \cap N$ $(\subset N_{\pr})$ and denote $\pi_N(N^{\diamond})$ by $M^{\diamond}$. 
		
		For \ref{i2.affineclosure}, since $Y$ and $N$ have FPIG, by Luna's slice theorem, the principal strata $Z_{\pr}$ and $M_{\pr}$ are non-singular such that the induced morphisms $Y_{\pr}\rightarrow Z_{\pr}$ and $N_{\pr}\rightarrow M_{\pr}$ are $G$-fiber bundles. In particular $M^{\diamond}$ is a dense open subset of $M$. Then the cotangent map $d^*\pi_Y$ of $\pi_Y$ induces an injective morphism $\pi_Y^*(T^*Z_{\pr}) \to T^*Y_{\pr}$ whose image is exactly
        $N^\diamond$. This gives the following
        commutative diagram with isomorphic horizontal morphisms:
		\[
		\begin{tikzcd}[row sep=large, column sep=large]
			\pi^*(T^*Z_{\pr}) \arrow[r,"{d^*\pi_Y}"] \arrow[d]
			     &   N^{\diamond}  \arrow[d,"{\pi_N}"] \\
			T^*Z_{\pr} \arrow[r,"\iota"]
			     &   M^{\diamond}.
		\end{tikzcd}
		\]
		More precisely, the map $\iota$ can be explicitly described as follows: for any $y\in Y_{{\rm pr}}$ with $z=\pi_Y(y)\in Z_{\pr}$ and  $\eta \in  T^*_z Z_{\pr}$, we define $\tilde{\eta} \in T_y^*Y$ by taking the composition 
		$$
		T_y Y \xlongrightarrow{d_y\pi_Y} T_z Z_{\pr} \xlongrightarrow{\eta}  \mathbb{C}.
		$$
		Note that $T_y O_y$ is contained in the kernel of $d_y\pi_Y$, we have $\tilde{\eta}(T_y O_y)=0$, which shows that $\tilde{\eta} \in N^{\diamond}$.
		Then one easily checks that the following morphism is a well-defined isomorphism (i.e. independent of the choice of $y \in \pi_Y^{-1}(z))$:
		\[
		\iota\colon T^*Z_{\pr} \longrightarrow M^{\diamond}, \quad \quad \eta \longmapsto \pi_N(\tilde{\eta}).
		\]
	    Since $Y$ is $2$-principal, the closed subset $Y\setminus Y_{\pr}$ has codimension at least two in $Y$, so that $Z_{\reg}\setminus Z_{\pr}$ has codimension at least two in $Z_{\reg}$. This implies that the restriction $\cO(T^*Z_{\reg})\rightarrow \cO(T^*Z_{\pr})$ is an isomorphism of $\bC$-algebras. So it remains to prove that $\codim M\setminus M^{\diamond}\geq 2$. Let $N^{r}\coloneqq p^{-1}(Y_{(r)})\cap N$ for any $0\leq r\leq \dim G$. Then, for $1\leq r\leq \dim G$, by \eqref{eq.dim-Nr}, we have
		\[
		\dim N^r \leq 2\dim Y - \dim G -2.
		\]
		On the other hand, since $Y$ is $2$-principal, we have
		\[
		\dim N^0 \setminus N^{\diamond} = \dim (Y_{(0)}\setminus Y_{\pr}) + \dim Y - \dim G \leq 2\dim Y - \dim G -2.
		\]
		This yields
		\[
		\dim N\setminus N^{\diamond} \leq 2\dim Y - \dim G -2 = \dim N -2,
		\]
		which implies $\dim(M\setminus M^{\diamond})\leq \dim M-2$, and hence \ref{i2.affineclosure} follows. 
		
		For \ref{i3.singularlocus}, firstly note that if $Z$ is non-singular, then $T^*Z$ is an irreducible non-singular affine variety such that $M=T^*Z$, so $M$ is also non-singular. 
		
		For the converse, we assume that $Z$ is singular. By \ref{i2.affineclosure}, there exists a natural projection $p_Z\colon M\rightarrow Z$ such that the restriction $p_Z|_{T^*Z_{\pr}}\colon T^*Z_{\pr}\rightarrow Z_{\pr}$ identifies to the natural projection. Let $\bfZ$ be the zero section of $p_Z$; that is, the closure of the zero section $\bfZ_{\pr}$ of $T^*Z_{\pr}\rightarrow Z_{\pr}$ in $M$ or equivalently $M^{\bC^*}$, where the $\bC^*$-action on $M$ is induced by the natural scaling on the fibers of $T^*Z_{\pr}\rightarrow Z_{\pr}$. Then the restriction $\iota=p_Z|_{\bfZ}\colon \bfZ\longrightarrow Z$ is an isomorphism. 
        
        Let $z$ be a singular point of $Z$ and let $\bfz\in \bfZ$ be the point such that $\iota(\bfz)=z$. Let $F$ be an irreducible component of the fiber of $p_Z$ passing through $\bfz$. Then $\dim F\geq \dim Z$ by the upper semi-continuity. Let $T_{z} Z$, $T_{\bfz} M$ and $T_{\bfz} F$ be the Zariski tangent spaces. Then we have the following complex of $\bC$-linear maps given by the tangent maps:
		\[
		T_{\bfz} F\longrightarrow T_{\bfz} M \longrightarrow T_{z} Z,
		\]
		where the first map is injective. Moreover, since $\iota\colon \bfZ\rightarrow Z$ is an isomorphism, the following composition 
		\[
		T_{\bfz} \bfZ \longrightarrow T_{\bfz} M \longrightarrow T_z Z
		\]
		is an isomorphism, so the last map is surjective. On the other hand, since $z$ is a singular point of $Z$, we have $\dim T_{z} Z > \dim Z$, which implies
		\[
		\dim T_{\bfz} M \geq \dim T_{\bfz} F + \dim T_{z} Z > \dim F + \dim Z \geq 2\dim Z = \dim M.
		\]
		Hence $M$ is singular at $\bfz$ by Jacobian's criterion.
	\end{proof}
	
	\begin{rem}
		For an irreducible non-singular affine $G$-variety $Y$ that is $2$-large, we have $Z_{\pr}=Z_{\reg}$ by \cite[Theorem 9.12]{Schwarz1995}. So $Z$ is non-singular if and only if $Z=Z_{\pr}$. Then Luna's slice theorem implies that $Y\rightarrow Z$ is actually a $G$-fiber bundle with fiber $G/H$ (see \cite[Proposition 6.4]{BrionKraftSchwarz2025}), where $(H)$ is the principal isotropy class.
	\end{rem}

\begin{example}
	Let $Y$ be an irreducible non-singular affine $G$-variety with FIPG. Assume that the identity component $G^0$ is a torus. By \cite[Proposition 10.1]{Schwarz1995}, $Y$ is $k$-large if and only if $Y$ is $k$-principal. If $G$ acts on $Y$ freely outside a locus of codimension at least two, then $Y$ is $2$-large. Theorem \ref{thm.main} implies that $T^*Y/\!\!/\!\!/G$ is a symplectic variety with terminal singularities that is the affine closure of $T^*(Y/\!/G)_{\pr}$. One particular example is that when $G$ is a finite group that acts freely on $Y^{\diamond} \subset Y$ with $\codim Y \setminus Y^{\diamond} \geq 2$, then $T^*Y/G$ is the affine closure of $T^*(Y^{\diamond}/G)$, as previously remarked in \cite[Example 2.8]{FuLiu2025}.
\end{example}
\subsection{Factoriality}

    Recall from \cite{KnopKraftVust1989} that for a $G$-variety $Y$, ${\rm Pic}_G(Y)$ denotes the set of isomorphism classes of line bundles on $Y$ with a $G$-linearization. Forgetting the $G$-linearization, we have the canonical homomorphism $\nu\colon {\rm Pic}_G(Y) \to {\rm Pic}(Y)$. The kernel ${\rm ker}(\nu)$ consists of $G$-linearizations of the trivial line bundle up to isomorphisms, which is denoted by $H^1_{\rm alg}(G, \mathcal{O}(Y)^\times)$, where $\cO(Y)^{\times}$ is the group of invertible regular functions on $Y$.

    For an irreducible non-singular affine $G$-variety $Y$, we remark that it follows from the standard properties of affine GIT that the natural projections $Y\rightarrow Z$ and $Y_{\pr}\rightarrow Z_{\pr}$ are both quotients by $G$ in the sense of \cite[4.1]{KnopKraftVust1989}.
	
	\begin{lem}
		\label{lem.Cl-quotient}
		Let $Y$ be an irreducible $2$-large non-singular affine $G$-variety.
        \begin{enumerate}
            \item If both $\Pic(Y)$ and $H^1_{\operatorname{alg}}(G,\cO(Y)^{\times})$ are finite, then $\Cl(Z)$ is finite and hence $Z$ is $\bQ$-factorial.

            \item If both $\Pic(Y)$ and $H^1_{\operatorname{alg}}(G,\cO(Y)^{\times})$ are trivial, then $\Cl(Z)$ is trivial and hence $Z$ is factorial. 
        \end{enumerate}
	\end{lem}
	
	\begin{proof}
		Since $Y$ is $2$-principal, it follows that $\codim(Y\setminus Y_{\pr})\geq 2$ for the $G$-stable open subset $Y_{\pr}$ of $Y$. In particular, one gets
		\[
		\Pic(Y_{\pr})\cong\Pic(Y)\quad \text{and}\quad \cO(Y)^{\times}=\cO(Y_{\pr})^{\times}.
		\]
        Moreover, since $Z\setminus Z_{\pr}$ has codimension at least two in $Z$, $Z_{\pr}\subset Z_{\reg}$ and $Z$ is normal, we have $\Cl(Z)\cong \Pic(Z_{\pr})$. By \cite[Lemma 2.2]{KnopKraftVust1989}, there exists an exact sequence
        $$
        1 \to H^1_{\rm alg}(G, \cO(Y_{\pr})^{\times}) \to {\rm Pic}_G(Y_{\pr}) \to {\rm Pic}(Y_{\pr}).
        $$ 
       Then applying \cite[Proposition 4.2]{KnopKraftVust1989} to the quotient $Y_{\pr}\rightarrow Z_{\pr}$ shows that  $\Pic(Z_{\pr})$ is
      a subgroup of  $\Pic_G(Y_{\pr})$, so the result follows.
	\end{proof}

\begin{proof}[Proof of Theorem \ref{thm.factoriality}]
	Firstly note that we have natural isomorphisms of groups
	\[
	\Cl(T^*Y/\!\!/\!\!/G) \cong \Pic(T^*Z_{\pr}) \cong \Pic(Z_{\pr}) \cong \Cl(Z),
	\]
	where the first isomorphism follows from Theorem \ref{thm.main}. By Lemma \ref{lem.Cl-quotient}, $\Cl(Z)$ is trivial (resp. finite) if both $H^1_{\operatorname{alg}}(G,\cO(Y)^{\times})$ and $\Pic(Y)$ are trivial (resp. finite). By assumption, either $E(Y)\coloneqq \cO(Y)^{\times}/\bC^{\times}$ is trivial or $G$ is connected, so $H^1(G/G^0, E(Y))$ is trivial. In particular, the following exact sequence from \cite[Proposition 2.3]{KnopKraftVust1989}
	$$
	\mathfrak{X}(G) \to H^1_{\operatorname{alg}}(G,\cO(Y)^{\times}) \to H^1(G/G^0, E(Y)).
	$$
	implies that $H^1_{\operatorname{alg}}(G,\cO(Y)^{\times})$ is trivial (resp. finite) if the character group $\mathfrak{X}(G)$ of $G$ is trivial (resp. finite) and hence if $G_{\operatorname{ab}}$ is trivial (resp. finite).
\end{proof}

\begin{proof}[Proof of Theorem \ref{thm.KLSconj-2large}]
    Recall that the singular locus of terminal singularities has codimension at least three and it follows from \cite[Corollary 1.3]{Fu03} that a locally $\mathbb{Q}$-factorial singular symplectic variety does not admit any symplectic resolution if its singular locus is not of pure codimension two. 

    Since $G$ is semisimple by assumption, it follows that $G$ is connected and $G_{\operatorname{ab}}$ is trivial, therefore $T^*Y/\!\!/\!\!/G$ is a $\bQ$-factorial symplectic variety with terminal singularities by Theorems \ref{thm.main} and \ref{thm.factoriality}. In particular, $T^*Y/\!\!/\!\!/G$ admits a symplectic resolution if and only if $T^*Y/\!\!/\!\!/G$ itself is non-singular, and hence if and only if $Y/\!/G$ is non-singular by Theorem \ref{thm.main}.
\end{proof}

Under some additional assumptions on $Y$, Theorem \ref{thm.factoriality} can be slightly improved.

\begin{prop}
	\label{prop.factoriality}
	Let $Y$ be an irreducible non-singular affine $G$-variety that is $2$-large. Assume that $\cO(Y)^{\times}=\bC^{\times}$,  $Y^G\not=\emptyset$ and $\Pic(Y)=0$. 
	\begin{enumerate}
		\item\label{i1.prop-factoriality} The quotient $M$ is factorial (resp. $\bQ$-factorial) if and only if $\Cl(Z)$ is trivial (resp. a torsion group).
		
		\item\label{i2.prop-factoriality} The divisor class group $\Cl(Z)$ is isomorphic to the kernel of $\fX(G)\rightarrow \fX(H)$, where $(H)$ is the principal isotropy class.
	\end{enumerate}
\end{prop}

\begin{proof}
	Applying \cite[Corollary 5.3]{KnopKraftVust1989} to the quotient $Y\rightarrow Z$ yields $\Pic(Z)=0$. Let $p_Z\colon M\rightarrow Z$ be the natural projection and let $\bfZ$ be the zero section of $p_Z$ with the isomorphism $\iota=p_Z|_{\bfZ}\colon \bfZ\longrightarrow Z$. Consider the following commutative diagram:
	\[
	\begin{tikzcd}[row sep=large, column sep=large]
		0=\Pic(\bfZ) 
		&  \Pic(M) \arrow[l] \arrow[r,hookrightarrow] 
		& \Cl(T^*Z_{\pr})=\Cl(M) \\
		& 0=\Pic(Z) \arrow[u,"{p_Z^*}"] \arrow[ul, "\cong" {below, sloped}, "{\iota^*}" above] \arrow[r,hookrightarrow]
		& \Cl(Z_{\pr}) = \Cl(Z) \arrow[u,"{p_Z^*}"].
	\end{tikzcd}
	\]
	Then one easily derives that for a Weil divisor $D$ on $Z$, the pull-back $p_Z^*D$ is Cartier (resp. $\bQ$-Cartier) if and only if $D$ is Cartier (resp. $\bQ$-Cartier) and if and only if $D$ is trivial (resp. a torsion) in $\Cl(Z)$, and hence \ref{i1.prop-factoriality} follows.
	
	Finally, as $\Cl(Z)=\Pic(Z_{\pr})$, we may apply \cite[Proposition 5.1 and Remark 4.5]{KnopKraftVust1989} to the quotient $Y_{\pr}\to Z_{\pr}$ to obtain the exact sequence
	\[
	1 \longrightarrow \Pic(Z_{\pr}) \longrightarrow \mathfrak{X}(G) \longrightarrow \mathfrak{X}(H),
	\] 
	which then gives the desired statement \ref{i2.prop-factoriality}.
\end{proof}

As an application, we prove the following partial converse to Theorem \ref{thm.factoriality} (for related results in the linear case, see \cite[Theorem 1.2 and Corollary 2.8]{BellamySchedler2019}).

\begin{cor}
	\label{cor.Conv-Thm-factoriality}
	Under the notation and assumptions of Proposition \ref{prop.factoriality}, the following hold:
	\begin{enumerate}
		\item The quotient $M$ is $\bQ$-factorial if and only if $G_{\operatorname{ab}}$ is finite.
		\item If $H$ is trivial, then $M$ is factorial if and only if $G$ is perfect.
	\end{enumerate}
\end{cor}

\begin{proof}
	By Theorem \ref{thm.factoriality} and Proposition \ref{prop.factoriality}, we only need to prove that if $M$ is $\bQ$-factorial, then $G_{\operatorname{ab}}$ is finite. Assume to the contrary that $G_{\operatorname{ab}}$ is not finite. Then we can choose a surjective character $\chi\colon G\rightarrow \bC^{\times}$. In particular, $\chi^n\not=1$ for all $n\geq 1$. Nevertheless, since the principal isotropy gruop $H$ is finite, there exists $m\in \bZ_{>0}$ such that $\chi^m|_H=1$, which then implies that $\Cl(Z)$ is not torsion and hence $M$ is not $\bQ$-factorial by Proposition \ref{prop.factoriality}, which is absurd.
\end{proof}

The following result generalizes \cite[Corollary 1.3]{BellamySchedler2019}.

\begin{cor}
	Under the notation and assumptions of Proposition \ref{prop.factoriality}, if $G_{\operatorname{ab}}$ is finite and $G$ acts non-trivially on $Y$, then $M$ does not admit a symplectic resolution.
\end{cor}

\begin{proof}
	By Theorem \ref{thm.main} and Theorem \ref{thm.factoriality}, it remains to show that $Z$ is singular. Indeed, by \cite[Theorem 9.12]{Schwarz1995}, we have $Z_{\reg}=Z_{\pr}$. In particular, if $Z$ is non-singular, then $Z_{\pr}=Z$ and $Y_{\pr}=Y$, and it follows that $G$ is a finite group acting trivially on $Y$ as $Y^G\not=\emptyset$, which is a contradiction.
\end{proof}

\subsection{Examples}  
\label{ss.Ex-Sharp}

We gather some examples to illustrate the sharpness of Theorems \ref{thm.main} and \ref{thm.factoriality}.

\begin{example}[\protect{\cite[Example 3.5]{BellamySchedler2019}}]
	Let $V=(\mathfrak{sl}_2)^2$, as a representation of $G=\operatorname{PGL}_2$, which is $1$-large, but not $2$-large. The Hamiltonian reduction $T^*V/\!\!/\!\!/G$ is identified with the locus of square zero matrices in $\mathfrak{sp}_4$, and in particular is a symplectic variety. Nevertheless, its singular locus is the codimension two locus consisting of rank one matrices in $\mathfrak{sp}_4$ and therefore it is not terminal, so the part \ref{i1.terminal} of Theorem \ref{thm.main} is optimal.
\end{example}

\begin{example}
	\label{ex.hypertoric} For $0 \leq k \leq n$, let $G=\bC^\times$ act on  $V_k=\bC^n$ by 
	$$\lambda \cdot (x_1, \dots, x_n) = (\lambda x_1, \dots, \lambda x_k, \lambda^{-1} x_{k+1}, \cdots, \lambda^{-1} x_n).$$ It is easy to see that for any $k$, the Hamiltonian reduction $T^*V_k/\!\!/\!\!/G$ is isomorphic to the minimal nilpotent orbit closure $\overline{\mathbf{O}}_{{\rm min}}$ in $\mathfrak{sl}_n$. 
	For $z\in V_k$, the $G$-orbit $Gz$ is closed in $V_k$ if and only if $z=0$ or $z \in V_k \setminus (F_k \cup F'_{n-k})$, where $F_k=\{x_1=\cdots=x_k=0\}$ and $F'_{n-k}=\{x_{k+1}=\cdots=x_n=0\}$. It follows that
	$(V_k)_{\pr} = V_k \setminus (F_k \cup F'_{n-k})$, so $V_k$ is 2-large if and only if $2\leq k \leq n-2$. In this case,
	Theorem \ref{thm.main} implies that $\overline{\mathbf{O}}_{{\rm min}}$ is the affine closure of $T^*((V_k/\!/G)_{\pr})$. On the other hand, $(V_k/\!/G)_{\pr}$ is just the punctured affine cone of the Segre variety
	$\bP^{k-1} \times \bP^{n-k-1} \subset \bP^{k(n-k)-1}$. This shows that $\overline{\mathbf{O}}_{{\rm min}}$ is the affine closure of the cotangent bundle of the punctured affine cone of a Segre variety, which
	is a special case of \cite[Theorem 1.4]{FuLiu2025} but with a new proof.  For $k=0$, $V_0/\!/G$ is just one point. For $k=1$, $V_1$ is $1$-large with 
	$V_1/\!/G \cong \bC^{n-1}$ and $(V_1/\!/G)_{\pr} \cong \bC^{n-1} \setminus \{0\}$. It follows that the affine closure of $T^*(V_1/\!/G)_{\pr}$ is $T^*\bC^{n-1}$ which is different from $\overline{\mathbf{O}}_{{\rm min}}$. Hence, parts  \ref{i2.affineclosure} and \ref{i3.singularlocus} of Theorem \ref{thm.main} are sharp.	
\end{example}

\begin{example}[\protect{\cite[Example 2.12]{BellamySchedler2019}}]
	Following the notation in Example \ref{ex.hypertoric},  we have  $G_{\operatorname{ab}}=G=\bC^{\times}$. Moreover, it is well-known that $\overline{\mathbf{O}}_{{\rm min}}$ is not $\bQ$-factorial and admits a symplectic resolution by collapsing the zero section of $T^*\bP^{n-1}$. 
	
	Now, let $n=2k$ and let $G'\coloneqq \bC^{\times}\rtimes \bZ_2$, where $s\in \bZ_2$ acts on $\bC^{\times}$ by $s(\lambda)=\lambda^{-1}$. Then $[G',G']=\bC^{\times}$ and $G_{\operatorname{ab}}\cong \bZ_2$. Let $s$ act on $V_k=\bC^{2k}$ as follows:
	\[
	s\cdot (x_1,\dots,x_k,x_{k+1},\dots,x_{2k}) = (x_{k+1},\dots,x_{2k},x_1,\dots,x_k).
	\]
    Then $V_k$ is a $2k$-large representation of $G'$ and $T^*V_k/\!\!/\!\!/G'$ is isomorphic to $\overline{\textbf{O}}_{\min}/\bZ_2$, which is isomorphic to the nilpotent orbit closure $\overline{\textbf{O}}_{[2^2,1^{2k-4}]}$ in $\mathfrak{sp}_{2k}$, or more precisely,
	\[
	\overline{\textbf{O}}_{[2^2,1^{2k-4}]} = \left\{A\in \mathfrak{sp}_{2k}\mid A^2=0,\,\rank(A)\leq 2\right\}
	\]
	The nilpotent closure $\overline{\textbf{O}}_{[2^2,1^{2k-4}]}$ is  next to the minimal orbit closure in the sense that it is contained in all other nilpotent orbit closures, and $\Cl(\overline{\textbf{O}}_{[2^2,1^{2k-4}]})\cong \bZ_2$ by \cite[Corollary 2.7]{Fu03}. These examples show that we cannot remove the triviality and finiteness conditions on $G_{\operatorname{ab}}$ in Theorem \ref{thm.factoriality}. In this sense, our Theorem \ref{thm.factoriality} is also optimal.
\end{example}

	\section{Differential operators on orbit spaces}
    \label{s.DO}
	
	Let $Z$ be an irreducible normal affine variety. Then the rings of differential operators on $Z$ and $Z_{\reg}$ coincide, i.e., $\cD(Z)=\cD(Z_{\reg})$. Let $S(Z)$ be the graded ring 
	\[
	S(Z)\coloneqq \bigoplus_{p\geq 0} H^0(Z_{\reg},\aS^p T_{Z_{\reg}}) \cong \cO(T^*Z_{\reg}).
	\]
    An element $\phi\in S^k(Z):=H^0(Z_{\reg},\aS^k T_{Z_{\reg}})$ is called an \emph{order $k$ symbol}. The ring $\cD(Z)$ has a filtration $\{\cD^k(Z)\}$ according to the order of differentiations, and the associated graded ring $\gr \cD(Z)$ is commutative. Then the order $k$ symbol map gives a linear map $\sigma_{Z,k}\colon \cD^{k}(Z)\rightarrow S^k(Z)$, which defines an injective homomorphism of graded algebras, called the \emph{graded symbol map}:
	\[
	\bar{\sigma}_Z\colon \gr\cD(Z) \longrightarrow S(Z).
	\]
	The \emph{symbol map} is the map
	\[
	\sigma_Z\colon \cD(Z)\longrightarrow S(Z),
	\]
	defined by $\sigma_Z(D)=\sigma_{Z,k}(D)$ if $D$ has order $k$.  An order $k$ symbol $\phi$ \emph{quantizes} into $D$ if $D\in \cD^k(Z)$ has order $k$ and the \emph{principal symbol} $\sigma_{Z,k}(D)$ of $D$ is $\phi$. If $Z$ is non-singular, then every symbol can be quantized, but for singular $Z$, the surjectivity of the symbol map is not well-understood.
	
	We will consider the case where $Z$ is a quotient, i.e. $Z=Y/\!/G$, where $Y$ is an irreducible non-singular affine $G$-variety. This problem was investigated in \cite{Schwarz1995}. Recall that there exists a canonical morphism $(\pi_Y)_*\colon \cD(Y)^G \rightarrow \cD(Z)$ that preserves the filtrations by order, where $\pi_Y\colon Y\rightarrow Z$ is the canonical quotient morphism. Then $(\pi_Y)_*$ is called \emph{graded surjective} if the induced graded morphism $\gr(\pi_Y)_*\colon \gr(\cD(Y)^G) \longrightarrow \gr(\cD(Z))$ is surjective.
	
	As an application of Theorem \ref{thm.main}, we give a short proof of the graded surjectivity of $2$-large quotients established in \cite{Schwarz1995}, and also show the surjectivity of the symbol map of such varieties.

	\begin{proof}[Proof of Theorem \ref{thm.Diff-Quot}]
		The shell $N\subset T^*Y$ is an irreducible normal affine variety such that $\cO(T^*Y)\rightarrow \cO(N)$ is surjective. Since $G$ is reductive, taking invariant subrings yields a surjection
		\[
		S(Y)^G = \cO(T^*Y)^G \longrightarrow \cO(N)^G.
		\]
		and an isomorphism
		\[
		\gr(\cD(Y)^G) \longrightarrow (\gr\cD(Y))^G.
		\]
		Since $Y$ is a non-singular affine variety, the graded symbol map $\bar{\sigma}_Y\colon \gr\cD(Y)\rightarrow S(Y)$ is an isomorphism, so we get a surjection
		\[
		\alpha\colon \gr(\cD(Y)^G) \xlongrightarrow{\cong} (\gr\cD(Y))^G \xlongrightarrow{\cong} S(Y)^G \longrightarrow \cO(N)^G.
		\]
		Thanks to Theorem \ref{thm.main}, there exists an canonical isomorphism
		\[
		 \iota\colon \cO(N)^G\cong \cO(T^*Y/\!\!/\!\!/G) \longrightarrow  S(Z_{\reg})=S(Z),
		\]
		and one easily checks that the following diagram is commutative:
		\[
		\begin{tikzcd}[row sep=large,column sep=large]
			\gr(\cD(Y)^G) \arrow[d,"{\gr(\pi_Y)_*}" left] \arrow[r,"\alpha"]
			    &  \cO(N)^G \arrow[d,"\iota"] \\
			\gr\cD(Z) \arrow[r,"\bar{\sigma}_Z"]
			    &  S(Z).
		\end{tikzcd}
		\]
		Since $\alpha$ is surjective and $\iota$ is an isomorphism, it follows that $\bar{\sigma}_Z$ is surjective and hence an isomorphism, which implies that $\gr(\pi_Y)_*$ is surjective.
        As $\cD(Z) \to {\rm gr} \cD(Z)$ is surjective, $\sigma_Z$ is also surjective.
	\end{proof}

\appendix	

\section{Hamiltonian slice theorem}
\label{s.HSThm}

	\subsection{Set-up} 
	
	 Let $G$ be a complex reductive algebraic group with Lie algebra $\fg$ and let $X$ be an affine quasi-Hamiltonian $G$-variety. Let $x\in N\cap X_{\reg}$ be a given point such that the orbit $O_x=Gx$ is closed. Let $H$ be the isotropy subgroup of $G$ at $x$ with Lie algebra $\fh$. As the orbit $O_x=Gx$ is closed, $H$ is reductive by Matsushima's theorem \cite{Matsushima1960}, and it follows from Luna's slice
	theorem (see \cite{Luna1973}) that there exists a locally closed non-singular affine $H$-stable subset
	$S\subset X$ containing $x$, called an \emph{\'etale slice}, such that the morphism 
	\[
	\varphi\colon Y\coloneqq G\times_H S\longrightarrow X,\quad [g,s]\longmapsto gs,
	\]
	is excellent and the image $GS\subset X$ is open affine and is contained in $X_{\rm reg}$ as $S$ is non-singular. 
	
	The symplectic form on $X_{\reg}$ can be pulled back to a symplectic form
	$\omega_Y$ on $Y$, making $(Y,\omega_Y)$ a non-singular Hamiltonian $G$-variety, and
	the following diagram is commutative:
	\[
	\begin{tikzcd}[row sep=large,column sep=large]
		Y \arrow[r,"\varphi"] \arrow[dr,"\mu_{Y}"]
		& X \arrow[d,"\mu_{X}"] \\
		& \fg^*,
	\end{tikzcd}
	\]
	where $\mu_X=\mu_{G,X}$ and $\mu_{Y}=\mu_{G,Y}$. Let $y=[e,x]\in Y$. Then $\varphi(y)=x$. 
	
	\begin{lem}
		\label{lem.T_yO_y-isotropic}
		The orbit $O_y\coloneqq Gy$ is isotropic with respect to $\omega_Y$. 
	\end{lem}
	
	\begin{proof}
		By assumption, $O_y$ is contained in $\mu_Y^{-1}(0)$, so the restricted Hamiltonian function $H_{\xi}|_{O_y}$ is a constant and the restriction of the Hamiltonian vector field $V_{\xi}$ to
		$O_y$ is tangent to $O_y$ for any $\xi\in \fg$. In particular, we have
		\[
		\omega_{Y}(V_{\xi},V_{\xi'})|_{O_y} = (dH_{\xi})|_{O_y} (V_{\xi'}|_{O_y}) = d(H_{\xi}|_{O_y})(V_{\xi'}|_{O_y}) = 0.
		\]
		for any $\xi$, $\xi'\in \fg$. Then the result follows since the $V_\xi$'s, $\xi\in \fg$,
		generate $T{O_y}$.
	\end{proof}
	
	\subsection{Auxiliary constructions}
	
	Let us consider the following composition of linear maps:
	\[
	\alpha\colon \fg  \longrightarrow  T_y Y  \xlongrightarrow{\omega_{Y,y}} T_y^*Y, \quad
	\xi        \longmapsto  V_{\xi}(y) \longmapsto dH_{\xi}(y).
	\]
	The image of $\fg$ in $T_y Y$ is exactly $T_y O_y$. Let $\fm_{Y,y}\subset
	\cO(Y)$ be the maximal ideal corresponding to $y$. Then $H_{\xi}\in \fm_{Y,y}$ for any
	$\xi\in \fg$ as $\mu_Y(y)=\mu_X(x)=0$. Moreover, under the canonical identification
	$T_y^*Y\cong \fm_{Y,y}/\fm_{Y,y}^2$, the linear map $\alpha$ is identified with the following map:
	\[
	\fg\longrightarrow \fm_{Y,y}/\fm_{Y,y}^2,\quad \xi \longmapsto [H_{\xi}].
	\] 
	Since $H$ is the isotropy subgroup and $\omega_{Y,y}$ is an isomorphism, the kernel
	of $\alpha$ is the subalgebra $\fh$. Since $H$ is reductive
	and $\fh$ is $H$-stable, we can choose an $H$-stable subspace $\fl$ of $\fg$ such
	that $\fg\cong \fh\oplus \fl$ and the induced linear map
	\[
	\alpha|_{\fl}\colon \fl \longrightarrow T_y Y \longrightarrow T_y^* Y
	\]
	is injective. 
	
	Let $S_y (\cong S)$ be the fiber of $Y=G\times_H S\rightarrow G/H$ over $[H]$. Then $y\in S_y$.
	Let $\fm_{S_y,y}\subset \cO(S_y)$ be the maximal ideal corresponding to $y$. Then we have a
	composition of linear maps
	\[
	\beta\colon \fg\longrightarrow T_y Y\longrightarrow T_y^* Y \cong \fm_{Y,y}/\fm_{Y,y}^2
	\longrightarrow T^*_y S_y \cong \fm_{S_y,y}/\fm_{S_y,y}^2,
	\]
	which sends $\xi\in \fg$ to $[H_{\xi}|_{S_y}]$.
	
	\begin{lem}
		\label{lem.Injectivity-restrict-to-S_y}
		The linear map $\beta|_{\fl}\colon \fl\rightarrow T_y^*S_y$ is injective.
	\end{lem}
	
	\begin{proof}
		Assume to the contrary that $\beta|_{\fl}$ is not injective. Then
		there exists a non-zero element $\xi\in \fl$ such that 
		\[
		dH_{\xi}(y)(T_y S_y) = \omega_{Y,y}(V_{\xi}(y), T_y S_y) = 0.
		\]
		On the other hand, since $0\not=V_{\xi}(y)\in T_y O_y$ and $O_y$ is isotropic with respect to $\omega_{Y,y}$ by Lemma \ref{lem.T_yO_y-isotropic}, one gets
		\[
		dH_{\xi}(y)(T_y O_y) = \omega_{Y,y}(V_{\xi}(y), T_y O_y) = 0.
		\]
		In particular, as $T_y Y\cong
		T_y S_y \oplus T_y O_y$, one derives 
		\[
		\omega_{Y,y}(V_{\xi}(y),T_y Y) = 0,
		\]
		which is a contradiction because $\omega_{Y,y}$ is non-degenerate.
	\end{proof}

	Let $L$ be the subspace of $\fm_{S_y,y}\subset \cO(S_y)$ generated by $H_{\xi}$ with
	$\xi\in \fl$. By Lemma \ref{lem.Injectivity-restrict-to-S_y}, we have $L\cap
	\fm_{S_y,y}^2=\{0\}$. In particular, since $L$ and $\fm^2_{S_y,y}$ are
	$H$-stable subspaces of $\fm_{S_y,y}$ and $H$ is reductive, we may choose a $H$-stable complement $F$ of
	$L\oplus \fm^2_{S_y}$ in $\fm_{S_y,y}$. Then the natural inclusion 
	\[
	F\oplus L\longrightarrow \fm_{S_y,y} \subset \cO(S_y)
	\]
	induces a dominant morphism $\nu\colon S_y\rightarrow (F\oplus L)^*=F^*\oplus
	L^*$ and a linear isomorphism
	\[
	T_y^* S_y \cong \fm_{S_y,y}/\fm_{S_y,y}^2 \cong F\oplus L.
	\]
	Clearly  $\nu$ is \'etale at $y$ and  maps the fixed point $y$ to $0\in F^*\oplus L^*$. Then Luna's fundamental lemma \cite{Luna1975} implies that $\nu$ is an excellent morphism in an open affine $H$-saturated neighborhood of $y$. Thus, we may assume that $\nu$ itself is excellent after shrinking $S_y$. Now, let $U$
	be the preimage $\nu^{-1}(F^*)$. Then $U$ is
	non-singular since $\nu$ is \'etale.
	
	\begin{lem}
		\label{lem.U-fiber}
		The variety $U$ is the fiber over $0\in \fl^*$ of the following natural
		composition:
		\[
		S_y \longrightarrow Y \xlongrightarrow{\mu_Y} \fg^* \longrightarrow \fl^*.
		\]
	\end{lem}
	
	\begin{proof}
		Note that $L$ is generated by the functions $H_{\xi}|_{S_y}$, $\xi\in \fl$. 
	\end{proof}
	
	\subsection{Proof of Theorem \ref{thm.Symplectic-Slice}}

	\begin{lem}
		\label{lem.non-degen-symp-form}
		The restriction of $\omega_Y$ to $U$ is symplectic in an open neighborhood of
		$y$.
	\end{lem}
	
	\begin{proof}
		By the openness of the non-degeneracy of $\omega_Y|_U$, it is enough to prove the non-degeneracy
		of $\omega_Y|_U$ at $y$.
		
		Consider the linear map $\alpha|_{\fl}\colon \fl\rightarrow T_y Y \rightarrow T_y^* Y$. Then the
		image of $\fl$ in $T_y Y$ is exactly $T_y O_y$. Let us identify $\fl$ with its
		image in $T_y^* Y$; that is, the subspace generated by $dH_{\xi}(y)$ for $\xi\in \fl$.
		Let $\Ann(\fl)\subset T_y Y$ be the annilhilator of $\fl$, which contains the tangent space $T_y U$ by the construction of $U$. In particular, since $T_y O_y$
		is isotropic with respect to $\omega_{Y,y}$, a dimension count then yields
		\[
		\Ann(\fl) = (T_y O_y)^{\perp_{\omega_{Y,y}}} = T_y O_y \oplus T_y U.
		\]
		The result follows because $T_y O_y$ is isotropic with respect to $\omega_{Y,y}$ by Lemma \ref{lem.T_yO_y-isotropic}.
	\end{proof}

	\begin{proof}[Proof of Theorem \ref{thm.Symplectic-Slice}]
		According to Lemma \ref{lem.non-degen-symp-form}, after removing the degeneracy locus of $\omega_Y|_U$ in $U$ if necessary, we may assume that $U\subset Y$ is a non-singular affine Hamiltonian $H$-variety such that the following diagram commutes:
		\[
		\begin{tikzcd}[row sep=large,column sep=large]
			Y \arrow[r,"{\mu_Y}"] 
			& \fg^* \arrow[d]  \\
			U  \arrow[r,"{\mu_U}"]  \arrow[u]
			& \fh^*
		\end{tikzcd}
		\]
		The closed subscheme $S_y\cap N_Y$ of $S_y$ is defined by the functions $H_{\xi}|_{S_y}$, $\xi\in \fg$. On the other hand, $U$ is defined by the functions $H_{\xi}|_{S_y}$, $\xi\in \fl$ and the commutative diagram above says that $N_U$ is defined by the functions $H_{\xi}|_U$, $\xi\in \fh$. As $\fg=\fh\oplus \fl$, the natural inclusion
		\[
		N_U^{\red} \longrightarrow (S_y\cap N_Y)^{\red}
		\]
		is actually an open embedding. Finally the statement follows from the fact that the restriction $\varphi|_{\varphi^{-1}(Z)}\colon \varphi^{-1}(Z)\rightarrow Z$ to any $G$-invariant closed subset $Z\subset X$ is again excellent.
	\end{proof}
	
	The following result is a by-product of the proof of Theorem \ref{thm.Symplectic-Slice}.
	
	\begin{cor}
		\label{cor.Hamiltonian-slice-linear}
		In Theorem \ref{thm.Symplectic-Slice}, there exists an \'etale slice $S$ containing $U$, a $G_x$-stable decomposition 
		\[
		T_x X = \fg/\fg_x \oplus T_x U \oplus (\fg/\fg_x)^*
		\]
		and a commutative diagram with excellent $G_x$-equivariant horizontal maps:
		\[
		\begin{tikzcd}[row sep=large, column sep=large]
			S \arrow[r,"\nu"]        &   T_x U\oplus (\fg/\fg_x)^* \\
			U \arrow[r,"{\nu|_U}"] \arrow[u,hookrightarrow]  &   T_x U \arrow[u,hookrightarrow].
		\end{tikzcd}
		\]
	\end{cor}
	
	\begin{proof}
		Since $\varphi$ is a $G$-equivariant \'etale morphism, it follows from the paragraph before Lemma \ref{lem.U-fiber} that there exists a $G_x$-stable decomposition
		\[
		T_x X \cong T_x O_x \oplus F^* \oplus L^*,
		\]
		such that $L^*\cong \fl^*$ and $F^*\cong T_x U$ as $G_x$-modules. So the statement follows from the fact that $T_x O_x$ and $\fl$ are isomorphic to $\fg/\fg_x$ as $G_x$-modules.
	\end{proof}
	
	\bibliographystyle{alpha}
	\bibliography{SQ}

\begin{thebibliography}{BLLT17}

\bibitem[AMM98]{AlekseevMalkinMeinrenken1998}
Anton Alekseev, Anton Malkin, and Eckhard Meinrenken.
\newblock Lie group valued moment maps.
\newblock {\em J. Differ. Geom.}, 48(3):445--495, 1998.

\bibitem[Bea00]{Beauville2000a}
Arnaud Beauville.
\newblock Symplectic singularities.
\newblock {\em Invent. Math.}, 139(3):541--549, 2000.

\bibitem[BKS25]{BrionKraftSchwarz2025}
Michel Brion, Hanspeter Kraft, and Gerald Schwarz.
\newblock Luna's {S}lice {T}heorem and applications.
\newblock {\em Bull. Amer. Math. Soc. (N.S.)}, 62(4):643--694, 2025.

\bibitem[BLLT17]{BuloisLehnLehnTerpereau2017}
Michael Bulois, Christian Lehn, Manfred Lehn, and Ronan Terpereau.
\newblock Towards a symplectic version of the {C}hevalley restriction theorem.
\newblock {\em Compos. Math.}, 153(3):647--666, 2017.

\bibitem[BS19]{BellamySchedler2019}
Gwyn Bellamy and Travis Schedler.
\newblock On symplecitc resolutions and factoriality of {H}amiltonian
  reductions.
\newblock {\em Math. Ann.}, 375:165--176, 2019.

\bibitem[FL25a]{FuLiu2025}
Baohua Fu and Jie Liu.
\newblock The affine closure of cotangent bundles of horospherical spaces.
\newblock {\em arXiv preprint arXiv:2502.06383}, 2025.

\bibitem[FL25b]{FuLiu2025a}
Baohua Fu and Jie Liu.
\newblock Symplectic singularities arising from algebras of symmetric tensors.
\newblock {\em J. Math. Pures Appl. (9)}, 204:Paper No. 103794, 2025.

\bibitem[Fle88]{Flenner1988}
Hubert Flenner.
\newblock Extendability of differential forms on nonisolated singularities.
\newblock {\em Invent. Math.}, 94(2):317--326, 1988.

\bibitem[Fu03]{Fu03}
Baohua Fu.
\newblock Symplectic resolutions for nilpotent orbits.
\newblock {\em Invent. Math.}, 151(1):167--186, 2003.

\bibitem[Har77]{Hartshorne1977}
Robin Hartshorne.
\newblock {\em Algebraic geometry}.
\newblock Springer-Verlag, New York-Heidelberg, 1977.
\newblock Graduate Texts in Mathematics, No. 52.

\bibitem[HSS20]{HerbigSchwarzSeaton2020}
Hans-Christian Herbig, Gerald~W. Schwarz, and Christopher Seaton.
\newblock Symplectic quotients have symplectic singularities.
\newblock {\em Compos. Math.}, 156(3):613--646, 2020.

\bibitem[HSS24a]{HerbigSchwarzSeaton2024a}
Hans-Christian Herbig, Gerald~W. Schwarz, and Christopher Seaton.
\newblock Isomorphisms of symplectic torus quotients.
\newblock {\em J. Symplectic Geom.}, 22(6):1179--1213, 2024.

\bibitem[HSS24b]{HerbigSchwarzSeaton2024}
Hans-Christian Herbig, Gerald~W. Schwarz, and Christopher Seaton.
\newblock When does the zero fiber of the moment map have rational
  singularities?
\newblock {\em Geom. Topol.}, 28(7):3475--3510, 2024.

\bibitem[Jun09]{Jung2009}
B.~Jung.
\newblock Ein symplektischer scheibensatz.
\newblock Diplomarbeit, Johannes Gutenberg-Universität Mainz, 2009.

\bibitem[KKV89]{KnopKraftVust1989}
Friedrich Knop, Hanspeter Kraft, and Thierry Vust.
\newblock The {P}icard group of a {$G$}-variety.
\newblock In {\em Algebraische {T}ransformationsgruppen und
  {I}nvariantentheorie}, volume~13 of {\em DMV Sem.}, pages 77--87.
  Birkh\"auser, Basel, 1989.

\bibitem[KLS06]{KaledinLehnSorger2006}
Dmitry Kaledin, Manfred Lehn, and Christoph Sorger.
\newblock Singular symplectic moduli spaces.
\newblock {\em Invent. Math.}, 164(3):591--614, 2006.

\bibitem[Lun73]{Luna1973}
Domingo Luna.
\newblock Slices \'etales.
\newblock In {\em Sur les groupes alg\'ebriques}, volume Tome 101 of {\em
  Suppl\'ement au Bull. Soc. Math. France}, pages 81--105. Soc. Math. France,
  Paris, 1973.

\bibitem[Lun75]{Luna1975}
Domingo Luna.
\newblock Adh\'erences d'orbite et invariants.
\newblock {\em Invent. Math.}, 29(3):231--238, 1975.

\bibitem[Mat60]{Matsushima1960}
Yoz\^o Matsushima.
\newblock Espaces homog\`enes de {S}tein des groupes de {L}ie complexes.
\newblock {\em Nagoya Math. J.}, 16:205--218, 1960.

\bibitem[Nam01]{Namikawa2001a}
Yoshinori Namikawa.
\newblock A note on symplectic singularities.
\newblock {\em arXiv preprint math/0101028}, 2001.

\bibitem[Sch95]{Schwarz1995}
Gerald~W. Schwarz.
\newblock Lifting differential operators from orbit spaces.
\newblock {\em Ann. Sci. \'Ecole Norm. Sup. (4)}, 28(3):253--305, 1995.

\end{thebibliography}
\end{document}